\documentclass[11pt]{amsart}

\usepackage[top=1.5in, bottom=1in, left=0.9in, right=0.9in]{geometry}

\usepackage{amscd,amsmath,amssymb,fancyhdr,color}
\usepackage[utf8]{inputenc}
\usepackage{amsfonts}
\usepackage{amsthm}
\usepackage{accents}
\usepackage{graphicx}
\usepackage{float}
\usepackage{subcaption}
\usepackage{verbatim}
\usepackage{indentfirst}
\usepackage{dsfont}
\usepackage{tikz}
\usepackage{tikz-cd}
\usetikzlibrary{matrix}
\usepackage[all]{xy}
\usepackage{enumerate}
\usepackage{extarrows}
\usepackage{csquotes}

\usepackage{scalerel,stackengine}
\stackMath
\newcommand\reallywidehat[1]{%
	\savestack{\tmpbox}{\stretchto{%
			\scaleto{%
				\scalerel*[\widthof{\ensuremath{#1}}]{\kern-.6pt\bigwedge\kern-.6pt}%
				{\rule[-\textheight/2]{1ex}{\textheight}}
			}{\textheight}%
		}{0.5ex}}%
	\stackon[1pt]{#1}{\tmpbox}%
}

\usepackage{faktor}

\usepackage{BOONDOX-uprscr}

\usepackage[backref=page]{hyperref}
\renewcommand*{\backref}[1]{}
\renewcommand*{\backrefalt}[4]{%
	\ifcase #1 (Not cited.)%
	\or        (Cited on page~#2.)%
	\else      (Cited on pages~#2.)%
	\fi}

\hypersetup{
	colorlinks   = true,
	citecolor    = magenta
}

\newcommand{\K}{K\"ahler}

\DeclareMathOperator{\Ker}{Ker}

\numberwithin{equation}{section}

\def\eqref#1{(\ref{#1})}

\newcommand{\Z}{{\mathbb Z}}
\newcommand{\C}{{\mathbb C}}
\newcommand{\R}{{\mathbb R}}

\renewcommand{\H}{{\mathbb H}}

\def\1{\sqrt{-1}\:}

\newcommand{\cntrct}                
{\hspace{2pt}\raisebox{1pt}{\text{$\lrcorner$}}\hspace{2pt}}

\newcommand{\Id}{\operatorname{Id}}

\renewcommand{\dim}{\operatorname{dim}}

\newcommand{\Img}{\operatorname{Im}}

\renewcommand{\Re}{\operatorname{Re}}
\renewcommand{\Im}{\operatorname{Im}}

\newcommand{\ie}{{\em i.e. }}
\newcommand{\eg}{{\em e.g. }}

\renewcommand{\to}{\longrightarrow}

\newcounter{Mycounter}[section]
\newcounter{lemma}[section]
\newcounter{claim}[section]
\newcounter{sublemma}[section]
\newcounter{corollary}[section]
\newcounter{theorem}[section]
\newcounter{conjecture}[section]
\newcounter{proposition}[section]
\newcounter{definition}[section]
\newcounter{example}[section]
\newcounter{remark}[section]
\newcounter{problem}[section]
\newcounter{question}[section]
\makeatletter

\@addtoreset{equation}{section}

\@addtoreset{footnote}{section}

\makeatother

\makeatletter
\DeclareRobustCommand*{\mfaktor}[3][]
{
   { \mathpalette{\mfaktor@impl@}{{#1}{#2}{#3}} }
}
\newcommand*{\mfaktor@impl@}[2]{\mfaktor@impl#1#2}
\newcommand*{\mfaktor@impl}[4]{
   \settoheight{\faktor@zaehlerhoehe}{\ensuremath{#1#2{#3}}}%
   \settoheight{\faktor@nennerhoehe}{\ensuremath{#1#2{#4}}}%
      \raisebox{-0.5\faktor@zaehlerhoehe}{\ensuremath{#1#2{#3}}}%
      \mkern-4mu\diagdown\mkern-5mu%
      \raisebox{0.5\faktor@nennerhoehe}{\ensuremath{#1#2{#4}}}%
}
\makeatother

\usetikzlibrary{arrows,chains,matrix,positioning,scopes}

\makeatletter
\tikzset{join/.code=\tikzset{after node path={%
			\ifx\tikzchainprevious\pgfutil@empty\else(\tikzchainprevious)%
			edge[every join]#1(\tikzchaincurrent)\fi}}}
\makeatother

\tikzset{>=stealth',every on chain/.append style={join},
	every join/.style={->}}

\makeatletter
\newtheorem*{rep@theorem}{\rep@title}
\newcommand{\newreptheorem}[2]{%
	\newenvironment{rep#1}[1]{%
		\def\rep@title{\ref{##1}}%
		\begin{rep@theorem}}%
		{\end{rep@theorem}}}
\makeatother

\newreptheorem{theorem}{Theorem}

\newtheoremstyle{named}{}{}{\itshape}{}{\bfseries}{.}{.5em}{\thmnote{#3's }#1}
\theoremstyle{named}

\begin{document}
	
	\newpage
	
	\title[Special non-K\" ahler metrics on Endo-Pajitnov manifolds]{Special non-K\" ahler metrics on Endo-Pajitnov manifolds}

        \author{Cristian Ciulic\u a}
	\address{Cristian Ciulic\u a  \newline
		\textsc{\indent University of Bucharest, Faculty of Mathematics and Computer Science\newline 
			\indent 14 Academiei Str., Bucharest, Romania}}
	\email{cristiciulica@yahoo.com}
	
	\author{Alexandra Otiman}
	\address{Alexandra Otiman \newline
		\textsc{\indent Institut for Matematik and Aarhus Institute of Advanced Studies, Aarhus University\newline 
			\indent 8000, Aarhus C, Denmark\newline
			\indent \indent and\newline
			\indent Institute of Mathematics ``Simion Stoilow'' of the Romanian Academy\newline 
			\indent 21 Calea Grivitei Street, 010702, Bucharest, Romania}}
	\email{alexandra.otiman@imar.ro; aiotiman@aias.au.dk}

	\author{Miron Stanciu}
	\address{Miron Stanciu \newline
		\textsc{\indent University of Bucharest, Faculty of Mathematics and Computer Science\newline 
			\indent 14 Academiei Str., Bucharest, Romania \newline
			\indent \indent and \newline
			\indent Institute of Mathematics ``Simion Stoilow'' of the Romanian Academy\newline 
			\indent 21 Calea Grivitei Street, 010702, Bucharest, Romania}}
	\email{miron.stanciu@fmi.unibuc.ro; miron.stanciu@imar.ro}
	
	\thanks{Cristian Ciulic\u a was partially supported by a grant of Ministry of Research and Innovation, CNCS - UEFISCDI, project no. PN-III-P4-ID-PCE-2020-0025, within PNCDI III.\\
		\indent Alexandra Otiman was partially supported by the European
Union’s Horizon 2020 research and innovation programme under the Marie Sk\l{}odowska-Curie grant agreement No 754513 and Aarhus University Research Foundation and by a grant of the Ministry of Research and Innovation, CNCS - UEFISCDI, project no. P1-1.1-PD-2021-0474.  \\
		\indent Miron Stanciu was partially supported by a grant of Ministry of Research and Innovation, CNCS - UEFISCDI, project no.
		PN-III-P1-1.1-TE-2021-0228, within PNCDI III. \\\\[.1cm]
	}
	
	\date{\today}

	\begin{abstract}
		We investigate the metric and cohomological properties of higher dimensional analogues of Inoue surfaces, that were introduced by Endo and Pajitnov. We provide a solvmanifold structure and show that in the diagonalizable case, they are formal and have invariant de Rham cohomology.  Moreover, we obtain an arithmetic and cohomological characterization of pluriclosed and astheno-K\" ahler metrics and show they give new examples in all complex dimensions.  
	\end{abstract}
	
	\maketitle
	
	\hypersetup{linkcolor=blue}

	\section{Introduction}
	
        While \K \ metrics yield powerful geometric results in complex analysis, their existence on compact manifolds is limited by strict topological and geometric constraints. This has motivated the investigation of alternative Hermitian metrics that still enable significant geometric insights. By either relaxing the cohomological or analytic requirements of the \K \ condition, a number of special Hermitian metrics have been proposed, whose existence might be canonical and lead to important analogous results from the \K \ setting. Among the most studied are \textit{balanced, locally conformally \K (lcK), pluriclosed} and \textit{astheno-\K}, the last two being particular cases of \textit{generalized Gauduchon} metrics. For more details about these geometries and their usefulness, we refer the reader to \cite{m82}, \cite{bel00}, \cite{OV_book}, \cite{bis89}, \cite{jostyau}, \cite{fgv19}, \cite{fww13}.

        In \cite{inoue}, Inoue constructed complex surfaces that have been widely studied for their interesting properties, among them the fact that they do not contain complex curves and do not admit \K \ metrics. This construction was generalized in a very well-known way by Oeljeklaus and Toma using algebraic number theory (\cite{ot}) to what become known as Oeljeklaus-Toma (OT) manifolds, a larger class of non-\K \ manifolds that sometimes (see \cite{dub14}, \cite{vuliOT}) admit lcK structures.
        
        In 2019 Endo and Pajitnov (\cite{pajitnov1}) proposed another generalization of Inoue surfaces to any dimension, that does not come from algebraic number theory, but only depends on an integer matrix $M$ with special requirements on its eigenvalues, just like the original construction. They proved that for most choices of $M$, these manifolds $T_M$ do not coincide with OT manifolds, instead being a new class of non-\K \ manifolds on which the existence of other Hermitian metrics like the ones mentioned above might be investigated. The authors only looked at the lcK case and in fact showed that for most matrices $M$, $T_M$ does not admit lcK metrics. 

        In this paper, our main goal was to find previously unknown examples of special non-\K \ manifolds using the Endo-Pajitnov constructions; to do this, we first proved that they are all solvmanifolds (but not nilmanifolds), computed all their Betti numbers and then used this to precisely describe the algebraic conditions that need to be imposed on the matrix $M$ in order for each of the types of metrics listed above to exist on $T_M$. These conditions lend themselves well to examples; we have included a method of constructing manifolds in any dimension that admit pluriclosed and astheno-\K \ metrics starting from an integer polynomial with a certain decomposition. Moreover, this class brings more evidence towards the recent conjectures of Fino-Vezzoni (see \cite{fv}) and Fino-Grantcharov-Verbitsky predicting an incompatibility between pluriclosed and balanced and pluriclosed, balanced and lcK metrics in complex dimension at least 3 (see \cite{fgv22}).  

        \medskip
        
        The paper is organized as follows.
        In Section \ref{sec:prelim}, we revisit the construction of the $T_M$-manifolds, remind a few properties about them that were proven in the original paper and also define the various types of non-\K \ metrics that we are interested in.
        In Section \ref{sec:solv}, we prove that all $T_M$ have a natural solvmanifold structure and also exhibit it explicitly with respect to $M$.
        In Section \ref{sec:coomologie}, we use the fact that all $T_M$ are mapping tori to compute all their Betti numbers. We find the precise algebraic conditions that must be imposed on the (eigenvalues of the) matrix $M$ in order for the higher Betti numbers which were not computed in \cite{pajitnov1} to be non-trivial.
        Finally, in Section \ref{sec:metrici}, we reach our main goal of investigating the existence of non-\K \ structures on Endo-Pajitnov manifolds, namely prove that they never admit balanced or lcK metrics but always admit lcb metrics, and find necessary and sufficient conditions for them to admit pluriclosed and astheno-\K \ metrics. We end with a new type of example that works in all dimensions and that, in particular, is not of OT type and carries a metric which is both pluriclosed and astheno-\K.

        \section{Preliminaries}
        \label{sec:prelim}
        
        \begin{definition}
            We start by listing the non-\K \ structures on an $n$-dimensional complex manifold whose existence will be of interest:
            \begin{itemize}
                \item {\it Locally conformally Kähler}, widely studied since I. Vaisman's paper  \cite{vai76}. A Hermitian metric $g$ is called {\it locally conformally K\"ahler (lcK)} if there exists a covering $(U_i)_i$ of the manifold and smooth functions $f_i$ on each $U_i$ such that $e^{-f_i}g$ is K\" ahler. The definition is conformally invariant and is equivalent to the existence of a closed one-form $\theta$ (called the {\em Lee form} such that the fundamental form $\omega$ of the metric $g$ satisfies $d\omega=\theta \wedge \omega$. For a recent comprehensive study of the development of lcK geometry, see \cite{OV_book}.
                
                \item {\it  Pluriclosed}, also referred to in the literature as {\em strongly K\"ahler with torsion}.  A Hermitian metric $g$ is called {\it pluriclosed} if its fundamental form $\omega$ satisfies $dd^c \omega=0$. See {\em e.g.}  \cite{bis89}.
                
                \item {\it  Astheno-\K}. A Hermitian metric $g$ is called {\it astheno-\K} if its fundamental form $\omega$ satisfies $dd^c\omega^{n-2}=0$. See \cite{jostyau}, \cite{fgv19}.
                
                \item {\it Balanced} (or {\em semi-K\"ahler}). A Hermitian metric $g$ is called {\it balanced} if its fundamental form $\omega$ satisfies $d\omega^{n-1}=0$, or equivalently if $\omega$ is co-closed. See  {\em e.g.} \cite{m82}.

                \item {\it Locally conformally balanced}. More generally, a Hermitian metric $g$ is called {\it locally conformally balanced (lcb)} if its fundamental form $\omega$ satisfies $d \omega^{n-1} = \theta \wedge \omega^{n-1}$ for a closed $1$-form $\theta$ (called the Lee form). lcK metrics are in particular examples of lcb. 
            \end{itemize}
        \end{definition}

        \medskip
        
        We now recall the construction of the Endo-Pajitnov manifolds from \cite{pajitnov1}.

        Let $n > 1$ and $M=(m_{ij})_{i, j} \in \textrm{SL}(2n+1, \mathbb{Z})$ such that the eigenvalues of $M$ are $\alpha, \beta_1, \cdots, \beta_k$, $\overline{\beta_1}, \cdots, \overline{\beta_k}$ with $\alpha>0, \alpha \neq 1$ and $\Im(\beta_j)>0$.
        
        Denote by $V$ the eigenspace corresponding to $\alpha$ and take $$W(\beta_j)=\{x \in \mathbb{C}^{2n+1}\ | \ \exists \ N \in \mathbb{N} \text{ such that } (M-\beta_j I)^Nx=0\}$$ and $W=\bigoplus \limits_{j=1}^{k} W(\beta_j)$, $\overline{W}=\bigoplus\limits_{j=1}^{k} W(\overline{\beta_j})$. Hence, $\mathbb{C}^{2n+1}=V \bigoplus W \bigoplus \overline{W}$.

        Let $a \in \R^{2n+1}$ be a non-zero eigenvector corresponding to $\alpha$ and take $\{b_1, \cdots, b_n\}$ a basis in $W$,
        \[
        a=(a^1, a^2,  \ldots,  a^{2n+1})^T, \ b_i=( b_{i}^1, b_{i}^2 \ldots,  b_{i}^{2n+1})^T, \ \forall i=\overline{1,n}.
        \]
        For any $1 \le i \le 2n + 1$, let $u_i = (a^i, b_1^i, ..., b_n^i) \in \R \times \C^n \simeq \R^{2n+1}$. Note that since $\{a, b_1, \ldots, b_n, \overline{b_1} \dots, \overline{b_n}\}$ is a basis in $\mathbb{C}^{2n+1}$, we have that $\{u_1, ..., u_{2n+1} \}$ are linearly independent over $\R$. 
        
        Let $f_M:W \longrightarrow W$ be the restriction of the multiplication with $M$ on $W$ and $R=(r_{ij})_{i, j}$ be the matrix of $f_M$ with respect to the basis $\{b_1, ..., b_n\}$. We consider the  automorphisms $g_0, g_1, \ldots, g_{2n+1}: \mathbb{H} \times \mathbb{C}^n \longrightarrow \mathbb{H} \times \mathbb{C}^n$,
        \[
            g_0(w,z)=(\alpha w, R^Tz), \ g_i(w,z)=(w,z)+u_i, \forall w \in \mathbb{H}, \forall z \in \mathbb{C}^n.
        \]
        Note that these are well defined because $\alpha > 0$ and the first component of $u_i$ is $a^i \in \R$.
        
        Let $G_M$ be the subgroup of Aut$(\mathbb{H} \times \mathbb{C}^n)$ generated by $g_0, g_1, \ldots, g_{2n+1}$. Pajitnov and Endo proved in \cite{pajitnov1} that the action of $G_M$ on $\mathbb{H} \times \mathbb{C}^n$ is free and properly discontinous. Hence, the quotient $T_M:= {G_M} \backslash (\mathbb{H} \times \mathbb{C}^n)$ is a compact complex manifold of complex dimension $n+1$, with $\pi_1(T_M) \simeq G_M$.

        \begin{remark}
        \label{rmk:nudepdeformaJordan}
            Note that the biholomorphism class of $T_M$ does not depend on the choice of basis $\{b_1, ..., b_n\}$.
        \end{remark}   

        \begin{remark}
            In the same paper, the authors prove that:
            \begin{itemize}
                \item The first Betti number $h^1(T^M) = 1$, so $T_M$ is non-\K.
                \item If $M$ is diagonalizable, then some $T_M$ are biholomorphic to OT manifolds (\cite[Proposition 5.3]{pajitnov1});
                \item If $M$ is not diagonalizable, then $T_M$ cannot be biholomorphic to any OT manifold (\cite[Proposition 5.6]{pajitnov1}) and does not admit lcK metrics (\cite[Proposition 4.6]{pajitnov1}).
            \end{itemize}
        \end{remark}            

        \bigskip

        \noindent{\bf Conventions.} Throughout the paper we shall use the conventions from \cite[(2.1)]{bes87} for the complex structure $J$ acting on complex forms on a complex manifold $(M, J)$. Namely:
        \begin{itemize}
            \item $J\alpha=i^{q-p} \alpha$, for any $\alpha \in \Lambda^{p, q}_{\mathbb{C}}M$, or equivalently $$J\eta(X_1, \ldots, X_p) = (-1)^p\eta(JX_1, \ldots, JX_p);$$
            \item the fundamental form of a Hermitian metric is given by $\omega (X, Y): = g(JX, Y)$;
            \item the operator $d^c$ is defined as $d^c := -J^{-1}dJ$, where $J^{-1} = (-1)^{\deg \alpha} J$. 
        \end{itemize}

        \section{The solvmanifold structure}
        \label{sec:solv}

       We shall endow $T_M$ with an explicit solvmanifold structure, which will later play an important role in finding special Hermitian metrics. This is done by identifying $\mathbb{H} \times \mathbb{C}^n$ with a Lie matrix group which is solvable and $G_M$ with a discrete subgroup. Readers familiar with OT manifolds will find many similarities between the constructions.

        First, note that the the biholomorphism class of $T_M$ does not depend on the choice of basis $\{b_1, ..., b_n\}$ described above. We then choose a basis such that $R^T$ is in the canonical Jordan form. As the eigenvalues $\beta_1, ..., \beta_n$ of $R$ sit above the real line, the logarithm of $R$ is well defined (see \eg \cite[Theorem 1.31]{higham}) and we can also define $R^t$ for any real power $t$ by the formula $R^t = \exp(t\log R)$ and still have $R^{t + s} = R^t \cdot R^s$ for any $t, s \in \R$.

        \begin{theorem}
            \label{th:solv}
            $T_M$ has a natural solvmanifold structure.

            Specifically, $T_M \simeq \mfaktor{\Gamma}{G}$, where $G$ is a solvable Lie group of dimension $2n + 2$, $\Gamma \le G$ is a discrete subgroup and the Lie algebra of $G$ has the structure equations
            
            $\mathfrak{g} = \langle A, X, Y_1, ..., Y_n, Y_{n+1}, ..., Y_{2n}\rangle_\R$
            \begin{equation}
                \label{ecstructura}
                \left\{
                \begin{split}
                    [A, X] &= \log \alpha X, \\
                    [A, Y_j] &= \sum\limits_{i \le j} \Re \Delta_{ij} Y_i + \sum\limits_{i \le j} \Im \Delta_{ij} Y_{n + i}, \ \forall j = \overline{1, n}, \\
                    [A, Y_{n + j}] &= -\sum\limits_{i \le j} \Im \Delta_{ij} Y_i + \sum\limits_{i \le j} \Re \Delta_{ij} Y_{n + i}, \ \forall j = \overline{1, n}, \\
                    [X, Y_j] &= 0, \ \forall j = \overline{1, 2n}, \\
                    [Y_j, Y_k] &= 0, \ \forall j, k = \overline{1, 2n},
                \end{split}
                \right.
            \end{equation}
            where $\Delta = \log R^T$ is an upper-triangular matrix and $\Delta_{jj} = \log \beta_j$. 
            
        \end{theorem}
        
        \begin{proof}

        Consider $G \subset \textrm{GL}_{2n+2}(\C)$,
        \[
        G =
        \left\{
            \left.
            A(x, t, z) = 
                \begin{pmatrix}
                \alpha^t & 0 & 0  & x  \\
                 0 & (R^t)^T & 0 &    z\\
                 0 & 0 & (\overline{R^t})^T &    \overline{z}\\
                 0 & 0 & 0  & 1
                \end{pmatrix}
            \right\vert
            t \in \mathbb{R}, x \in \mathbb{R}, z^T \in \mathbb{C}^n
        \right\}.
        \]
        
        By elementary computations, $A(x, t, z) \cdot A(y, s, \zeta) = A(\alpha^t y + x,t + s, (R^t)^T \zeta + z)$ and therefore, $G$ is a subgroup of $\textrm{GL}_{2n+2}(\C)$. As it is clearly also a submanifold of $\textrm{GL}_{2n+2}(\C)$, it is a Lie group. Since $\varphi: \H \times \C^n \to G, \varphi(x + i \alpha^t, z) = A(x, t, z)$ is a diffeomorphism, we endow $\H \times \C^n$ with the induced group structure. To see that $G$ (and thus $\H \times \C^n$) is solvable, the only thing left is to note that, by our choice of basis such that $R^T$ is in Jordan form, $G$ contains only upper-triangular matrices, hence so does its Lie algebra and therefore it follows easily that it is solvable.

        We will now express the action of $G_M$ on $\H \times \C^n$ as the action of a discrete subgroup $\Gamma$ on $G$. Indeed, take
        \[
        \Gamma = 
        \left\{
            \left.
            \gamma(m, W) = 
                \begin{pmatrix}
                \alpha^m & 0 & 0  & \  \\
                 0 & (R^m)^T & 0 &  P \cdot W \\
                 0 & 0 & (\overline{R^m})^T &  \ \\
                 0 & 0 & 0  & 1
                \end{pmatrix}
            \right\vert
            m \in \Z, W^T \in \Z^{2n+1}
        \right\},
        \]
        where $P=\begin{pmatrix} a \\ b_1 \\ \vdots \\ b_n \\ \overline{b_1} \\ \vdots \\ \overline{b_n} \end{pmatrix} \in \mathcal{M}_{2n+1}(\C)$. 
        Take also $B:=\begin{pmatrix}
        b_1^{1} & \cdots  & b_n^{1}  \\
        \vdots &     &    \vdots\\
         b_1^{2n+1} & \cdots  & b_n^{2n+1}
        \end{pmatrix}$. To prove that $\Gamma$ is a subgroup, we make use of the easily verifiable equality $MB = BR$ (see \cite[Lemma 2.2]{pajitnov1}).

        Indeed, for $\gamma(m_1, W_1), \gamma(m_2, W_2) \in \Gamma$,
        \[\
        \gamma(m_1, W_1) \cdot \gamma(m_2, W_2) = 
            \begin{pmatrix}
                \alpha^{m_1 + m_2} & 0 & 0 & \ \\
                0 & (R^{m_1 + m_2})^T & 0 & Y \\
                0 & 0 & \overline{(R^{m_1 + m_2})^T} & \ \\
                0 & 0 & 0 & 1
            \end{pmatrix},
        \]
        where 
        \begin{equation*}
            \begin{split}
                Y &= \begin{pmatrix}
                        \alpha^{m_1} & 0 & 0 \\
                        0 & (R^{m_1})^T & 0 \\
                        0 & 0 & \overline{(R^{m_1})}^T 
                    \end{pmatrix} \cdot PW_1 + P \cdot W_2 \\ &= 
                    \begin{pmatrix}
                        M^{m_1} \cdot a \\ (R^{m_1})^T \cdot B^T \\ \overline{(R^{m_1})^T \cdot B^T}
                    \end{pmatrix} \cdot W_1 + P \cdot W_2 =
                    \begin{pmatrix}
                        a^T \cdot (M^{m_1})^T \\ B^T \cdot (M^{m_1})^T \\ \overline{B^T \cdot (M^{m_1})^T}
                    \end{pmatrix} \cdot W_1 + P \cdot W_2 \\
                    &= P\cdot M^{m_1} \cdot W_1 + P \cdot W_2 = P \cdot (M^{m_1} W_1 + W_2),
            \end{split}
        \end{equation*}
        so 
        \begin{equation}
            \label{legegrupgamma}
            \gamma(m_1, W_1) \cdot \gamma(m_2, W_2) = \gamma(m_1 + m_2, M^{m_1} W_1 + W_2)
        \end{equation}
        As $M$ has integer coefficients, $M^{m_1} W_1 + W_2 \in \Z^{2n + 1}$ as required.

        Additionally, there is a natural correspondence $\Phi: G_M \to \Gamma$, where $G_M = \langle g_0, g_1, ..., g_{2n + 1} \rangle$ is the group of automorphisms of $\H \times \C^n$ that defines $T_M$, given by $\Phi (g_0) = \gamma(1, 0)$ and $\Phi (g_i) = \gamma(0, f_i^T)$, where $f_i$ is just the $i$th element of the canonical basis of $\Z^{2n+1}$. The reader can easily check that $\Phi$ is a group isomorphism using \eqref{legegrupgamma}.
        
        Finally, we have a biholomorhism $T_M \simeq \mfaktor{\Gamma}{G}$ and thus, a solvmanifold structure, since one can check that $\varphi: \H \times \C^n \to G$ defined above is equivariant with respect to the actions of $G_M$ and $\Gamma$ respectively (identified via $\Phi$). 


        As for the structure equations of $\mathfrak{g}$, we use the variation of the $2n + 2$ real parameters of $G$ to obtain the generators:
        \begin{equation*}
            \begin{split}
                A &:= \frac{d}{ds}_{|s = 0} A(0, s, 0) = 
                \begin{pmatrix}
                    \log \alpha & 0 & 0  & 0  \\
                     0 & \log R^T & 0 &    0\\
                     0 & 0 & \overline{\log R^T} & 0\\
                     0 & 0 & 0  & 0
                \end{pmatrix}, \\
                X &:= \frac{d}{ds}_{|s = 0} A(s, 0, 0) = 
                \begin{pmatrix}
                     0 & 0 & 0 & 1  \\
                     0 & 0 & 0 & 0\\
                     0 & 0 & 0 & 0\\
                     0 & 0 & 0 & 0
                \end{pmatrix}, \\
                Y_j &:= \frac{d}{ds}_{|s = 0} A(0, 0, s\cdot e_j) = 
                \begin{pmatrix}
                     0 & 0 & 0 & 0  \\
                     0 & 0 & 0 & e_j^T \\
                     0 & 0 & 0 & e_j^T \\
                     0 & 0 & 0 & 0
                \end{pmatrix}, 
                \ 1 \leq j \leq n, \\
                Y_{n + j} &:= \frac{d}{ds}_{|s = 0} A(0, 0, s\cdot ie_j) = 
                \begin{pmatrix}
                     0 & 0 & 0 & 0  \\
                     0 & 0 & 0 & ie_j^T \\
                     0 & 0 & 0 & -ie_j^T \\
                     0 & 0 & 0 & 0
                \end{pmatrix}, \ 1 \leq j \leq n, \\
            \end{split}
        \end{equation*}

        where $\{ e_1, ..., e_n \} $ is the canonical basis of $\C^n$. The structure equations \eqref{ecstructura} follow immediately.       
        \end{proof}

        \begin{remark}
            From the description above, it follows by \cite[Theorem 1.2]{kasuya} that, in the case where $M$ is diagonalizable, $T_M$ is a formal manifold (see also \cite[Example 5]{kasuya}), in the sense of Sullivan.
        \end{remark}
        
        \section{de Rham cohomology of $T_M$}
        \label{sec:coomologie}

        In order to compute the Betti numbers of $T_M$, we crucially use the fact that it is a mapping torus (\cite[Proposition 2.9]{pajitnov1}):
        
        \begin{proposition}
            $T_M$ is diffeomorphic to the mapping torus of $\mathbb{T}^{2n+1}$ with the gluing map $M^T : \mathbb{T}^{2n+1} \to \mathbb{T}^{2n+1}$.
        \end{proposition}

        \bigskip

        We employ the following notation in what follows:

        \begin{definition}
            If $V$ is a vector space and $f: V \to V$ is linear, we denote by $f^{\wedge k} : \wedge^k V \to \wedge^k V$ the $k$-exterior power of $f$ \ie 
            \[
            f^{\wedge k} (v_1 \wedge ... \wedge v_k) = f(v_1) \wedge ... \wedge f(v_k).
            \]
        \end{definition}

        The following fact is straightforward:
        
        \begin{remark}
            \label{valpropriiprodusexterior}
            If $\dim V = n$ and $\lambda_1, ..., \lambda_n$ are the eigenvalues of $f$, then the eigenvalues of $f^{\wedge k}$ are $\lambda_{i_1} \lambda_{i_2} ...\lambda_{i_k}$ for all $1 \le i_1 < i_2 < ... < i_k \le n$.
        \end{remark}

        \bigskip

        It turns out that the easiest way to give the Betti numbers of $T_M$ is to express them in terms of the eigenvalues of the exterior powers of $M$:
        
        \begin{theorem}
        \label{nrbetti}
            For any $1 \le k \le 2n + 1$, $h^k(T_M) = g_{k - 1} + g_k$, where $g_k$ is the geometric multiplicity of $1$ as an eigenvalue of $M^{\wedge k}$. 

            In particular:
            \begin{itemize}
                \item $h^1(T_M) = 1$, recovering \cite[Lemma 3.1]{pajitnov1}.
                \item $h^k(T_M) = 0$, $\forall \ 1 < k < 2n + 1$, for a generic $M$, \ie if no product of some, but not all eigenvalues of $M$ is $1$.
            \end{itemize}
        \end{theorem}

        \begin{proof}
            For any mapping torus $Z$ given by $f: X \to X$ and any $k$, we have the short exact sequence
            \begin{equation}
            \label{sirexactcoomologie}
            0 \longrightarrow \faktor{H^{k-1}(X)}{\Img (\Id - f_{k-1}^*)} \longrightarrow H^k(Z) \longrightarrow \Ker (\Id - f_k^*) \longrightarrow 0,          
            \end{equation}
            where $f_k^*:H^k(X) \to H^k(X)$ is the mapping induced by $f$ (see \eg \cite[Example 2.48]{hatcher}).

            We apply this for $f: \mathbb{T}^{2n + 1} \to \mathbb{T}^{2n+1}$ being the multiplication by $M^T$. We have 
            \[
            H^1 (\mathbb{T}^{2n + 1}) \simeq \bigoplus\limits_{j = 1}^{2n + 1} H^1(S^1) \text{    and    } H^k (\mathbb{T}^{2n + 1}) \simeq \bigwedge\nolimits^{k} \left( \bigoplus\limits_{j = 1}^{2n + 1} H^1(S^1) \right), 1 \le k \le {2n + 1},          
            \]
            where this is to be taken as an exterior sum \ie $H^1(S^1) \wedge H^1(S^1) = 0$ only if the two appear in the same position, because they are $1$-dimensional. Then $H^k (\mathbb{T}^{2n + 1}) \simeq \bigwedge\nolimits^{k} H^1 (\mathbb{T}^{2n + 1})$ and we only need to compute $f_1^*$, as by the properties of the pullback, $f_k^* = (f_1^*)^{\wedge k}$. But again by the definition of the pullback and the fact that $f$ is just the multiplication by $M^T$, we have that 
            \[
            f_1^* : \bigoplus\limits_{j = 1}^{2n + 1} H^1(S^1) \to \bigoplus\limits_{j = 1}^{2n + 1} H^1(S^1)
            \]
            is just the multiplication by $M$ with respect to the canonical basis of $H^1(S^1)$.
            
            It follows that $\dim \Ker (\Id - f_k^*) = g_k$ and, coming back to \eqref{sirexactcoomologie}, we have
            \[
            h^k(T_M) = \dim \Ker (\Id - f_{k-1}^*) + \dim \Ker (\Id - f_k^*) = g_{k-1} + g_k,
            \]
            for any $k \ge 1$.
        \end{proof}

        Again note that while this result tells us that all previously unknown Betti numbers are $0$ in the generic case, it also gives the precise algebraic properties that the matrix $M$ must satisfy in order for the manifold $T_M$ to have non-trivial higher cohomology. 

        \bigskip

        In the case when $M$ is diagonalizable, one can actually give left invariant representatives for the cohomology classes, which means that we have an isomorphism $H^{k}_{dR}(T_M) \simeq H^k(\mathfrak{g})$, for any $1 \leq k \leq 2n+2$.

        \begin{corollary}
            If $M$ is diagonalizable, with the notations from \ref{nrbetti}, we have
            \[
                H^k(T_M, \C) = \left( \bigoplus\limits_{\substack{1 \le i_1 < ... < i_{k-1} \le 2n+1 \\ \lambda_{i_1} \cdot ... \cdot \lambda_{i_{k-1}} = 1 }} \C \frac{d \Im w}{\Im w} \wedge e_{i_1} \wedge ... \wedge e_{i_{k-1}} \right) \bigoplus \left( \bigoplus\limits_{\substack{1 \le i_1 < ... < i_k \le 2n+1 \\ \lambda_{i_1} \cdot ... \cdot \lambda_{i_k} = 1 }} \C e_{i_1} \wedge ... \wedge e_{i_k} \right),
            \]
        \end{corollary}   
        where $e_1=\frac{d \Re w}{\Im w}$, $e_{i+1}=dz_i$ and $e_{i+n+1}=d\overline{z}_i$, for $1 \leq i \leq n$ and $\lambda_1=\alpha$, $\lambda_{i+1}=\beta_i$ and $\lambda_{i+n+1}=\overline{\beta}_i$, for $1 \leq i \leq n$. 

        \begin{remark}
            See also \cite[Example 7.2]{bf23} for an example of a $4$-dimensional manifold of $T_M$ type with an explicit computation of its cohomology.
        \end{remark}

        \section{Non-\K \ metrics on $T_M$}
        \label{sec:metrici}

        We investigate in this section the existence of special metrics on the $T_M$ manifolds. To this aim, we use the structure equations given in \eqref{ecstructura} to describe a $(1, 0)$-coframe for the Lie algebra $\mathfrak{g}$:

        \begin{proposition}
            There exists a $(1,0)$-coframe $\mathfrak{g}^* = \langle \eta, \theta_1, ..., \theta_n \rangle_\C$ such that
            \begin{equation}
                \label{ecstructuradiff}
                \left\{
                \begin{split}
                    d \eta &= \log \alpha \ \eta \wedge \overline{\eta}, \\
                    d \theta_k &= -\sum\limits_{j \ge k} \Delta_{kj} (\eta + \overline{\eta}) \wedge \theta_j, \ 1\leq k \leq n,
                \end{split}
                \right.
            \end{equation}
            where $\Delta = \log R^T$ is an upper-triangular matrix and $\Delta_{jj} = \log \beta_j$. 
        \end{proposition}

        \begin{proof}
            Remember that $G \simeq \H \times \C^n$, from which it inherits its complex structure. Namely, on $\mathfrak{g} = \langle A, X, Y_1, ..., Y_n, Y_{n+1}, ..., Y_{2n}\rangle_\R$ we have $J X = A$ and $J Y_j = Y_{n + j}$ for any $j = \overline{1, n}$. Thus,
            \[
            \mathfrak{g} = \langle A + i X, Y_{n + 1} + i Y_1, ..., Y_{2n} + i Y_n \rangle_\C
            \]
            is a basis of $(1, 0)$-vectors. Now choose the $(1,0)$-coframe given by the dual basis: $\mathfrak{g}^* = \langle \eta, \theta_1, ..., \theta_n \rangle_\C$. We want to express $d\eta, d\theta_1, ..., d\theta_n$ with respect to this coframe.

            For any $(1, 0)$-form $\sigma \in \mathfrak{g}^*$, using the Cartan formula $d\sigma(U, V) = -\sigma([U, V])$, we have
            \begin{equation*}
                \begin{split}
                    d\sigma = &- \sigma([A + iX, A - iX]) \cdot \eta \wedge \overline{\eta} - \sum\limits_{j = 1}^n \sigma([A + iX, Y_{n + j} + iY_j]) \cdot \eta \wedge \theta_j \\ 
                      &- \sum\limits_{j = 1}^n \sigma([A + iX, Y_{n + j} - iY_j]) \cdot \eta \wedge \overline{\theta_j} - \sum\limits_{j = 1}^n \sigma([A - iX, Y_{n + j} + iY_j]) \cdot \overline{\eta} \wedge \theta_j \\
                      &- \sum\limits_{j, l = 1}^n \sigma([Y_{n + j} + iY_j, Y_{n + l} + i Y_l]) \cdot \theta_j \wedge \theta_l - \sum\limits_{j, l = 1}^n \sigma([Y_{n + j} + iY_j, Y_{n + l} - i Y_l]) \cdot \theta_j \wedge \overline{\theta_l}.
                \end{split}
            \end{equation*}
            Note that the last two sums vanish for any $\sigma$ by \eqref{ecstructura}.
            
            If $\sigma = \eta$, all terms except the first one vanish by \eqref{ecstructura} and the definition of $\eta$, and 
            \[
            \eta([A + iX, A - iX]) = \eta(-2i\log \alpha X) = -\log \alpha \ \eta ( (A + iX) - (A - iX))) = -\log \alpha,
            \]
            so indeed, $d\eta = \log \alpha \ \eta \wedge \overline{\eta}$.

            On the other hand, in the expression of $d\theta_k$ the term containing $\eta \wedge \overline{\eta}$ vanishes and
            \begin{equation*}
                \begin{split}
                    \theta_k([A \pm iX, Y_{n + j} + i Y_j]) &= \theta_k([A, Y_{n + j} + i Y_j]) \\ 
                    &= \theta_k \left( -\sum\limits_{l \le j} \Im \Delta_{lj} Y_l + \sum\limits_{l \le j} \Re \Delta_{lj} Y_{n + l} + i \sum\limits_{l \le j} \Re \Delta_{lj} Y_l + i\sum\limits_{l \le j} \Im \Delta_{lj} Y_{n + l} \right) \\
                    &= \theta_k \left( \sum\limits_{l \le j} \Re \Delta_{lj} (Y_{n + l} + i Y_l) + \sum\limits_{l \le j} i\Im \Delta_{lj} (Y_{n + l} + i Y_l) \right) \\
                    &= \sum\limits_{l \le j} \Delta_{lj} \theta_k (Y_{n + l} + i Y_l) = \Delta_{kj}
                \end{split}
            \end{equation*}
            while 
            \begin{equation*}
                \begin{split}
                    \theta_k([A \pm iX, Y_{n + j} - i Y_j]) &= \theta_k([A, Y_{n + j} - i Y_j]) \\ 
                    &= \theta_k \left( -\sum\limits_{l \le j} \Im \Delta_{lj} Y_l + \sum\limits_{l \le j} \Re \Delta_{lj} Y_{n + l} - i \sum\limits_{l \le j} \Re \Delta_{lj} Y_l - i\sum\limits_{l \le j} \Im \Delta_{lj} Y_{n + l} \right) \\
                    &= \theta_k \left( \sum\limits_{l \le j} \Re \Delta_{lj} (Y_{n + l} - i Y_l) - \sum\limits_{l \le j} i\Im \Delta_{lj} (Y_{n + l} - i Y_l) \right) \\
                    &= \sum\limits_{l \le j} \overline{\Delta_{lj}} \theta_k (Y_{n + l} - i Y_l) = 0
                \end{split}
            \end{equation*}
            In the end, we get $d\theta_k = -\sum\limits_{j \ge k} \Delta_{kj} (\eta + \overline{\eta}) \wedge \theta_j$ as stated.            
        \end{proof}

        We also record separately the special case where $M$ is diagonal:

        \begin{corollary}   
            If the matrix $M$ is diagonal, then there is a $(1,0)$-coframe $\mathfrak{g}^* = \langle \eta, \theta_1, ..., \theta_n \rangle_\C$ such that
            \begin{equation}
                \label{ecstructuradiffdiag}
                \left\{
                \begin{split}
                    d \eta &= \log \alpha \ \eta \wedge \overline{\eta}, \\
                    d \theta_k &= -\log \beta_k \ (\eta + \overline{\eta}) \wedge \theta_k, \ \forall k = \overline{1, n}.
                \end{split}
                \right.
            \end{equation}
        \end{corollary}

\bigskip

We will make heavy use of these structural equations to study the existence of special metrics, starting with balanced and locally conformally balanced metrics, for which the answers do not depend on the algebraic properties of the matrix $M$.
        
\begin{proposition} Endo-Pajitnov manifolds do not admit balanced metrics.
\end{proposition}

\begin{proof} Assume $\Omega$ is a balanced metric on $T_M$. The $(n, n)$-form $\Omega^{n}$ is then strictly positive and closed and moreover, $\Omega':=\pi^*\Omega^{n}$ shares the same properties on the intermediate cover $\overline{T_M}=\mathbb{H} \times \mathbb{C}^{n}/ \langle g_1, \ldots, g_n \rangle$. Since $\overline{T_M}$ is diffeomorphic to $\mathbb{T}^{2n+1} \times \mathbb{R}_{+}$, we can define the average of $\Omega'$ with respect to the torus action by
\begin{equation*}
\tilde{\Omega}:=\int_{\mathbb{T}^{2n+1}} a^*\Omega'd\mu (a),
\end{equation*}
where $\mu$ is the constant volume form on $\mathbb{T}^{2n+1}$ given by $\int_{\mathbb{T}^{2n+1}} d\mu =1$. The form $\tilde{\Omega}$ is thus constant in the variables ${z_1, \ldots, z_n}$ and moreover, $G_M$-invariant. Indeed, using the invariance of $\Omega'$, we get
\begin{align*}
    g_0^*\tilde\Omega&=\int_{\mathbb{T}^{2n+1}} (ag_0)^*\Omega'd\mu (a)=\int_{\mathbb{T}^{2n+1}} (g_0c_{g_0}(a))^*\Omega'd\mu (a)\\ &=\int_{\mathbb{T}^{2n+1}} c_{g_0}(a)^*\Omega'd\mu (c_{g_0}(a))=\tilde\Omega.
\end{align*}
By lifting now $\tilde{\Omega}$ to $\mathbb{H} \times \mathbb{C}^{n}$ we get a strictly positive, closed, $G_M$-invariant $(n, n)$-form, that is constant in the variables $\{z_1, \ldots, z_n\}$. This means we can split $\tilde{\Omega}$ as a sum
\begin{equation*}
    \tilde{\Omega}=f(w)\mathrm{i} dz_1 \wedge \ldots \wedge dz_n \wedge d\overline{z_1} \wedge \ldots \wedge d\overline{z_n}+ \tilde{\Omega}_0,
\end{equation*}
where $i_{\frac{\partial}{\partial w}}\Omega_0 \neq 0$, $i_{\frac{\partial}{\partial \overline{w}}}\Omega_0 \neq 0$. Notice that $i_{\frac{\partial}{\partial z_1}}\ldots i_{\frac{\partial}{\partial z_n}}\tilde{\Omega}_0 = 0$ and since $\tilde{\Omega}_0$ is constant in the variables $\{z_1, \ldots, z_n\}$, we deduce $i_{\frac{\partial}{\partial z_1}}\ldots i_{\frac{\partial}{\partial z_n}}d\tilde{\Omega}_0=0$. Moreover, by the positivity of $\tilde{\Omega}$, $f>0$.  However, $d\tilde{\Omega}=0$, which gives
\begin{equation*}
i_{\frac{\partial}{\partial z_1}}\ldots i_{\frac{\partial}{\partial z_n}} \left(\frac{\partial f}{\partial w} dw + \frac{\partial f}{\partial \overline{w}} d\overline{w} \right) \wedge dz \wedge d\overline{z}=0, 
\end{equation*}
where $dz:=dz_1 \wedge \ldots \wedge dz_n$, and further implies that $f$ is constant. Since $\tilde{\Omega}$ is $g_0$-invariant and $i_{\frac{\partial}{\partial z_1}}\ldots i_{\frac{\partial}{\partial z_n}}g_0^*\tilde{\Omega}_0=0$, we infer $g_0^*dz \wedge d\overline{z}=dz \wedge d\overline{z}$ and hence, $|\mathrm{det} R|=1.$ This is impossible, however, since $\alpha \neq 1$ and therefore, no balanced metric can exist on $T_M$.
\end{proof}

\medskip

However, for any matrix $M$, the manifold $T_M$ admits locally conformally balanced metrics:

\begin{proposition}
    With the notations from \eqref{ecstructuradiff}, the metric
    \[
    \omega = i \left( \eta \wedge \overline{\eta} + \theta_1 \wedge \overline{\theta_1} + ... + \theta_n \wedge \overline{\theta_n} \right)  
    \]
    is lcb.
\end{proposition}

\begin{proof}
    This is a straightforward computation. Indeed,
    \[
    \omega^n = i^n \left( n! \ \theta_1 \wedge \overline{\theta_1} \wedge ... \wedge \theta_n \wedge \overline{\theta_n} + (n-1)! \sum\limits_{j = 1}^n \eta \wedge \overline{\eta} \wedge \theta_1 \wedge \overline{\theta_1} \wedge ... \wedge \reallywidehat{ \theta_j \wedge \overline{\theta_j} } \wedge ... \wedge \theta_n \wedge \overline{\theta_n} \right),
    \]
    so, using \eqref{ecstructuradiff},
    \begin{equation*}
        \begin{split}
            d\omega^n &= -i^n n! \sum\limits_{j=1}^n \left(\Delta_{jj} + \overline{\Delta_{jj}} \right) (\eta + \overline{\eta}) \wedge \theta_1 \wedge \overline{\theta_1} \wedge ... \wedge \theta_n \wedge \overline{\theta_n} \\
            &= \left( -\sum\limits_{j=1}^n \left(\Delta_{jj} + \overline{\Delta_{jj}} \right) (\eta + \overline{\eta}) \right) \wedge \omega^n.
        \end{split}
    \end{equation*}

    Obviously the Lee form $\theta = -\sum\limits_{j=1}^n \left(\Delta_{jj} + \overline{\Delta_{jj}}\} \right) (\eta + \overline{\eta})$ is closed, so $\omega$ is indeed lcb.
\end{proof}

\bigskip

Complementing the result by Endo and Pajitnov that the manifolds $T_M$ do not admit lcK metrics for a non-diagonalizable $M$ (\cite[Proposition 5.6]{pajitnov1}), we can use the structural equations \eqref{ecstructuradiffdiag} to show that, in fact, they do not admit lcK metrics in any case:

\begin{proposition}
    Endo-Pajitnov manifolds do not admit lcK metrics.
\end{proposition}

\begin{proof}
    We only need to look at the case where $M$ is diagonalizable. Assume $\Omega$ is an lcK metric on $T_M$ \ie $d \Omega = \Theta \wedge \Omega$ for a real closed $1$-form $\Theta$. Since $H^1_{dR}(T_M) \simeq H^1(\mathfrak{g})$, one may choose $\Theta$ to be left-invariant and by the averaging procedure in \cite[Theorem 7]{bel00}, we may take $\Omega$ to be left-invariant and make all the computations necessary on the Lie algebra, say $\omega \in \bigwedge^2(\mathfrak{g}^*)$ and $\theta \in \mathfrak{g}^*$.

    As $\omega$ is a $(1, 1)$-form, we can write it as
    \begin{equation}
        \omega = \sum\limits_{i = 1}^n c_i \eta \wedge \overline{\theta_i} + \sum\limits_{i = 1}^n d_i \overline{\eta} \wedge \theta_i + \sum\limits_{i, j = 1}^n \overline{a_{ij}} \theta_i \wedge \overline{\theta_j} + b \eta \wedge \overline{\eta},
    \end{equation}
    where, because $\omega = \overline{\omega}$, we have $d_i = \overline{c_i}$ and $a_{ij} = -\overline{a_{ji}}$. 

    Then
    \[
    d\omega = \sum\limits_{i = 1}^n c_i \log (\alpha \overline{\beta_i} )\eta \wedge \overline{\eta} \wedge \overline{\theta_i} - \sum\limits_{i = 1}^n d_i \log (\alpha \beta_i) \eta \wedge \overline{\eta} \wedge \theta_i - \sum\limits_{i, j = 1}^n \overline{a_{ij}} \log (\beta_i \overline{\beta_j}) (\eta + \overline{\eta}) \wedge \theta_i \wedge \overline{\theta_j}.
    \]

    On the other hand, any real closed $1$-form has the expression
    \[
    \theta = p (\eta + \overline{\eta}) + \sum\limits_{i=1}^n q_i (\theta_i + \overline{\theta_i}),
    \]
    so by identifying the terms in the equality $d \omega = \theta \wedge \omega$, we get, for instance,
    \[
    p = \log(\beta_i \overline{\beta_j}), \ \forall 1 \le i, j \le n, 
    \]
    which is obviously a contradiction.
\end{proof}

\bigskip

We now turn to the study of pluriclosed and astheno-\K \ metrics. The situation in which the matrix $M$ is diagonalizable is completely described in the result below:

\begin{theorem}
\label{th:Mdiag}
Let $T_M$ be the manifold associated to a diagonalizable $M \in \textrm{SL}(2n+1, \mathbb{Z})$. The following statements are equivalent:
\begin{enumerate}[(1)]
    \item $T_M$ admits a pluriclosed metric.
    \item There is $1 \leq i_0 \leq n$ such that $\alpha\beta_{i_0}\overline{\beta}_{i_0}=1$ and $|\beta_i|=1$, for any $1 \leq i \leq n$, $i \neq i_0$.
    \item $T_M$ admits an astheno-K\" ahler metric.
    \item The characteristic polynomial of $M$ splits in $\mathbb{Z}[X]$ as $P_M=f_0h$, where $f_0$ is a monic polynomial of degree 3 and $h$ is a self-reciprocal polynomial.
    \item $h^2(T_M) = n - 1$.
\end{enumerate}
\end{theorem}

\begin{proof}
    We shall prove first that $(2) \Rightarrow (1)$ and $(2) \Rightarrow (3)$. We can assume $i_0 = 1$.
    
    Take $\tilde{\Omega}=\mathrm{i}\frac{dw \wedge d\overline{w}}{(\Im w)^2}+ \Im w dz_1 \wedge d\overline{z}_1 + \sum^n_{j} \mathrm{i}dz_j \wedge d\overline{z}_j$ on $\mathbb{H} \times \mathbb{C}^n$. Then it is straighforward to verify that $\tilde{\Omega}$ is $G_M$-invariant and satisfies $dd^c \tilde{\Omega}^k=0$, for any $1 \leq k \leq n$. In particular, $\tilde{\Omega}$ defines a metric on $T_M$ that is both pluriclosed and astheno-K\" ahler. 

    \medskip

    For $(1) \Rightarrow (2)$, we will make use of the structure equations \eqref{ecstructuradiff}. Assume $\Omega$ is a pluriclosed metric on $T_M$ \ie $dJd \Omega = 0$. As in the lcK case, by an averaging procedure similar to \cite[Theorem 7]{bel00}, we may take $\Omega$ to be left-invariant and make all the computations necessary on the Lie algebra, say $\omega \in \bigwedge^2(\mathfrak{g}^*)$.

    As $\omega$ is a $(1, 1)$-form, we can write it as
    \begin{equation}
        \label{omegacoordonate}
        \omega = \sum\limits_{i = 1}^n c_i \eta \wedge \overline{\theta_i} + \sum\limits_{i = 1}^n d_i \overline{\eta} \wedge \theta_i + \sum\limits_{i, j = 1}^n \overline{a_{ij}} \theta_i \wedge \overline{\theta_j} + b \eta \wedge \overline{\eta},
    \end{equation}
    where, because $\omega = \overline{\omega}$, we have $d_i = \overline{c_i}$ and $a_{ij} = -\overline{a_{ji}}$. 
        
    We introduce the following notations to ease the computations that follow: for any two vectors of differential forms $\alpha \in \left( \bigwedge^p(\mathfrak{g}^*) \right)^n$ and $\beta \in \left( \bigwedge^q(\mathfrak{g}^*) \right)^n$, let $$\langle \alpha, \beta \rangle = \sum\limits_{i = 1}^n \alpha_i \wedge \overline{\beta_i} \in \bigwedge\nolimits^{p + q}(\mathfrak{g}^*).$$ This is $\C$-linear in the first term, $\C$-antilinear in the second and $\langle \beta, \alpha \rangle = (-1)^{pq} \overline{\langle \alpha, \beta \rangle}$. Additionally, $d \langle \alpha, \beta \rangle = \langle d \alpha, \beta \rangle + (-1)^{\deg \alpha} \langle \alpha, d\beta \rangle$ and, for a matrix $P \in \mathcal{M}_n(\C), \langle P\cdot \alpha, \beta \rangle  = \langle \alpha, \overline{P}^T \cdot \beta \rangle$.

    Taking $A = (a_{ij})_{i, j}$, $v = (c_i)_i$ and $\theta = (\theta_i)_i$, we can rewrite the above expression as 
    \begin{equation}
        \label{omegaprodscalar}
        \omega = \langle \eta v, \theta \rangle - \langle \theta, \eta v \rangle + \langle \theta, A \cdot \theta \rangle + b \eta \wedge \overline{\eta}.
    \end{equation}
    Note that as $a_{ij} = -\overline{a_{ji}}$, we have $\langle A \cdot \alpha, \beta \rangle = - \langle \alpha, A \cdot \beta \rangle$.

     We first take the calculations as far as we can for the general structure equations \eqref{ecstructuradiff}, without using the fact that $M$ is diagonal, as we will make use of them later. If we denote $E = - \Delta$, we have $d \theta = (\eta + \overline{\eta}) \wedge E \cdot \theta$, hence
    \begin{equation*}
        \begin{split}
            d \omega &= \log \alpha \langle (\eta \wedge \overline{\eta}) v, \theta \rangle - \langle \eta v, (\eta + \overline{\eta}) \wedge E \cdot \theta \rangle - \langle (\eta + \overline{\eta}) \wedge E \cdot \theta, \eta v \rangle + \log \alpha \langle \theta, (\eta \wedge \overline{\eta}) v\rangle \\
            &+ \langle (\eta + \overline{\eta}) \wedge E \cdot \theta, A \cdot \theta \rangle - \langle \theta, (\eta + \overline{\eta}) \wedge A E \cdot \theta \rangle,
        \end{split}
    \end{equation*}
    or, after canceling the terms containing $\eta \wedge \eta$ and some rearranging using the properties of the pairing $\langle \ , \ \rangle$ stated above,
    \begin{equation*}
        \begin{split}
            d \omega &= \log \alpha \langle (\eta \wedge \overline{\eta}) v, \theta \rangle + \log \alpha \langle \theta, (\eta \wedge \overline{\eta}) v\rangle \\
            &- \langle \eta v, \eta \wedge E \cdot \theta \rangle - \langle \eta \wedge E \cdot \theta, \eta v \rangle \\
            &+ \langle A \cdot \theta, (\eta + \overline{\eta}) \wedge E \cdot \theta \rangle + \langle (\eta + \overline{\eta}) \wedge E \cdot \theta, A \cdot \theta \rangle.
        \end{split}
    \end{equation*}
    Note that in the expression above, each term is the conjugate of the other term written on the same line. Then 
    \begin{equation*}
        \begin{split}
            Jd \omega &= \log \alpha \langle (\eta \wedge \overline{\eta}) v, \theta \rangle - \overline{\log \alpha \langle (\eta \wedge \overline{\eta}) v, \theta \rangle} \\
            &+ \langle \eta \wedge E \cdot \theta, \eta v \rangle - \overline{\langle \eta \wedge E \cdot \theta, \eta v \rangle} \\
            &+ \langle A \cdot \theta, (\eta - \overline{\eta}) E \cdot \theta \rangle - \overline{\langle A \cdot \theta, (\eta - \overline{\eta}) E \cdot \theta \rangle}.
        \end{split}
    \end{equation*}

    Finally, on differentiating the above expression, the first line vanishes immediately and 
    \begin{equation*}
        \begin{split}
            dJd \omega &= \log \alpha \langle \eta \wedge \overline{\eta} \wedge E \cdot \theta, \eta v \rangle - \langle \eta \wedge (\eta + \overline{\eta}) E^2 \cdot \theta, \eta v \rangle + \log \alpha \langle \eta \wedge E \cdot \theta, \eta \wedge \overline{\eta} \rangle \\
            &- \left( (\overline{\log \alpha \langle \eta \wedge \overline{\eta} \wedge E \cdot \theta, \eta v \rangle - \langle \eta \wedge (\eta + \overline{\eta}) E^2 \cdot \theta, \eta v \rangle + \log \alpha \langle \eta \wedge E \cdot \theta, \eta \wedge \overline{\eta} \rangle} \right) \\
            &+ \langle (\eta + \overline{\eta}) \wedge A E \cdot \theta, (\eta - \overline{\eta}) E \cdot \theta \rangle - 2\log \alpha \langle A \cdot \theta, \eta \wedge \overline{\eta} \wedge E \cdot \theta \rangle + 2\langle A \cdot \theta, \eta \wedge \overline{\eta} \wedge E^2 \cdot \theta \rangle \\
            &- \left( \overline{\langle (\eta + \overline{\eta}) \wedge A E \cdot \theta, (\eta - \overline{\eta}) E \cdot \theta \rangle - 2\log \alpha \langle A \cdot \theta, \eta \wedge \overline{\eta} \wedge E \cdot \theta \rangle + 2\langle A \cdot \theta, \eta \wedge \overline{\eta} \wedge E^2 \cdot \theta \rangle} \right).
        \end{split}
    \end{equation*}
    Again the first two lines above vanish and we have
    \begin{equation*}
        \begin{split}
            dJd \omega &= 2 \langle \eta \wedge \overline{\eta} \wedge A E \cdot \theta, E \cdot \theta \rangle - 2 \log \alpha \langle A \cdot \theta, \eta \wedge \overline{\eta} \wedge E \cdot \theta \rangle + 2 \langle A \cdot \theta, \eta \wedge \overline{\eta} \wedge E^2 \cdot \theta \rangle \\
            &- \left( \overline{2 \langle \eta \wedge \overline{\eta} \wedge A E \cdot \theta, E \cdot \theta \rangle - 2 \log \alpha \langle A \cdot \theta, \eta \wedge \overline{\eta} \wedge E \cdot \theta \rangle + 2 \langle A \cdot \theta, \eta \wedge \overline{\eta} \wedge E^2 \cdot \theta \rangle} \right) \\
            &= 2 \eta \wedge \overline{\eta} \wedge \left( \langle A E \cdot \theta, E \cdot \theta \rangle - \log \alpha \langle A \cdot \theta, E \cdot \theta \rangle + \langle A \cdot \theta, E^2 \cdot \theta \rangle \right) \\
            &-2 \eta \wedge \overline{\eta} \wedge \overline{\left( \langle A E \cdot \theta, E \cdot \theta \rangle - \log \alpha \langle A \cdot \theta, E \cdot \theta \rangle + \langle A \cdot \theta, E^2 \cdot \theta \rangle \right)}.
        \end{split}
    \end{equation*}

    In the end, we get that $\omega$ is pluriclosed if any only if 
    \begin{equation*}
        \begin{split}
            &\langle A E \cdot \theta, E \cdot \theta \rangle - \log \alpha \langle A \cdot \theta, E \cdot \theta \rangle + \langle A \cdot \theta, E^2 \cdot \theta \rangle \\
            - &\langle E \cdot \theta, A E \cdot \theta \rangle + \log \alpha \langle E \cdot \theta, A \cdot \theta \rangle - \langle E^2 \cdot \theta, A \cdot \theta \rangle = 0,
        \end{split}
    \end{equation*}
    or, using the properties of $A$, the pairing $\langle \ , \rangle$ and reverting back to $\Delta = -E$,
    \begin{equation*}
        \langle (2\overline{\Delta}^T A \Delta + (\overline{\Delta}^T)^2 A + A \Delta^2 + \log \alpha( \overline{\Delta}^T A + A\Delta) ) \cdot \theta, \theta \rangle = 0.
    \end{equation*}

    As $(\theta_i \wedge \overline{\theta_j})_{i, j = \overline{1, n}}$ is a basis, this is equivalent to
    \begin{equation}
        \label{pluriclosed}
        2\overline{\Delta}^T A \Delta + (\overline{\Delta}^T)^2 A + A \Delta^2 + \log \alpha( \overline{\Delta}^T A + A\Delta) = 0.
    \end{equation}

    Now remember that according to our hypothesis, $M$ is diagonal, so $\Delta$ is diagonal with $\Delta_{jj} = \log \beta_j$. The pluriclosed condition \eqref{pluriclosed} then simplifies to 
    \begin{equation*}
        \begin{split}
            a_{ji} \left( 2\log \beta_i \ \overline{\log \beta_j} + (\log \beta_i)^2 + \overline{(\log \beta_j)}^2 + \log \alpha (\log \beta_i + \overline{\log \beta_j}) \right) = a_{ji} \log(\alpha \beta_i \overline{\beta_j}) \log(\beta_i \overline{\beta_j})  = 0, 
        \end{split}
    \end{equation*}
    for all $i, j = \overline{1, n}$.

    But as $\omega$ is Hermitian, we have $\omega(X, JX) > 0$ for all real $X \neq 0$. Then, from \eqref{omegacoordonate},
    \[
        \overline{a_{jj}} = \omega(Y_{n + j} + i Y_j, Y_{n + j} - i Y_j) = 2i \omega(Y_j, Y_{n + j}) = 2i \omega (JY_{n + j}, Y_{n + j}), 
    \]
    hence $a_{jj} \neq 0$ and we get that for any $1 \le j \le n$, $| \beta_j |^2 = 1 $ or $\alpha |\beta_j|^2 = 1$. However, since $\alpha |\beta_1|^2 ... |\beta_n|^2 = \det(M) = 1$ and $\alpha \neq 1$, it is obvious that this can only happen if there is one $i_0$ such that $\alpha |\beta_{i_0}|^2 = 1$ and $|\beta_j|^2 = 1$ for all other $j \neq i_0$.

    \medskip

    For the implication $(3) \Rightarrow (2)$, we proceed in a similar way: Let $\Omega$ be an astheno-\K \ metric on $T_M$. Again by an averaging procedure as in \cite[Theorem 7]{bel00} of the $(n-1, n-1)$-form $\Omega^{n-1}$, we obtain a positive left-invariant $(n-1, n-1)$-form $\omega$, which is $\partial\overline{\partial}$-closed.

    Since $\omega$ is an $(n-1, n-1)$-form, we can write it as
    \begin{equation*}
        \begin{split}
            \omega &= \sum\limits_{i, k = 1}^n a_{i,k} \ \theta_1 \wedge ... \wedge \widehat{\theta_i} \wedge ... \wedge \theta_n \wedge \overline{\theta_1} \wedge ... \wedge \widehat{\overline{\theta_k}} \wedge ... \wedge \overline{\theta_n} \\
            &+ \sum\limits_{i, j, k = 1}^n b_{ij,k} \ \eta \wedge \theta_1 \wedge ... \wedge \widehat{\theta_i} \wedge ... \wedge \widehat{\theta_j} \wedge ... \wedge \theta_n \wedge \overline{\theta_1} \wedge ... \wedge \widehat{\overline{\theta_k}} \wedge ... \wedge \overline{\theta_n} \\
            &+ \overline{\left( \sum\limits_{i, j, k = 1}^n b_{ij,k} \ \eta \wedge \theta_1 \wedge ... \wedge \widehat{\theta_i} \wedge ... \wedge \widehat{\theta_j} \wedge ... \wedge \theta_n \wedge \overline{\theta_1} \wedge ... \wedge \widehat{\overline{\theta_k}} \wedge ... \wedge \overline{\theta_n} \right)} \\ 
            &= \sum\limits_{i, k = 1}^n a_{i,k} \ \theta_1 \wedge ... \wedge \widehat{\theta_i} \wedge ... \wedge \theta_n \wedge \overline{\theta_1} \wedge ... \wedge \widehat{\overline{\theta_k}} \wedge ... \wedge \overline{\theta_n} \\
            &+ 2 \Re \sum\limits_{i, j, k = 1}^n b_{ij,k} \ \eta \wedge \theta_1 \wedge ... \wedge \widehat{\theta_i} \wedge ... \wedge \widehat{\theta_j} \wedge ... \wedge \theta_n \wedge \overline{\theta_1} \wedge ... \wedge \widehat{\overline{\theta_k}} \wedge ... \wedge \overline{\theta_n},
        \end{split}
    \end{equation*}
    where, because $\omega = \overline{\omega}$, we have $\overline{a_{i,k}} = (-1)^{n-1} a_{k, i}$. In order to compress the expressions, we'll use the notations
    \begin{equation*}
        \begin{split}
            \theta_{i, k} &= \theta_1 \wedge ... \wedge \widehat{\theta_i} \wedge ... \wedge \theta_n \wedge \overline{\theta_1} \wedge ... \wedge \widehat{\overline{\theta_k}} \wedge ... \wedge \overline{\theta_n} \\
            \theta_{ij,k} &= \theta_1 \wedge ... \wedge \widehat{\theta_i} \wedge ... \wedge \widehat{\theta_j} \wedge ... \wedge \theta_n \wedge \overline{\theta_1} \wedge ... \wedge \widehat{\overline{\theta_k}} \wedge ... \wedge 
            \overline{\theta_n},
        \end{split}
    \end{equation*}
    so 
    \begin{equation}
    \label{omegan-1coordonate}
        \begin{split}
            \omega = \sum\limits_{i, k = 1}^n a_{i,k} \theta_{i, k} + 2 \Re \sum\limits_{i, j, k = 1}^n b_{ij,k} \ \eta \wedge \theta_{ij, k}.
        \end{split}
    \end{equation}
    Then, using \eqref{ecstructuradiffdiag}, we have
    \begin{equation*}
        \begin{split}
            d \theta_{i, k} &= - \log \left( \frac{\beta_1 \beta_2 ... \beta_n \overline{\beta_1} ... \overline{\beta_n} }{\beta_i \overline{\beta_k}} \right) (\eta + \overline{\eta}) \wedge \theta_{i, k} = \log \left( \alpha \beta_i \overline{\beta_k} \right) (\eta + \overline{\eta}) \wedge \theta_{i, k} \\
            d \theta_{ij,k} &= - \log \left( \frac{\beta_1 \beta_2 ... \beta_n \overline{\beta_1} ... \overline{\beta_n} }{\beta_i \beta_j \overline{\beta_k}} \right) (\eta + \overline{\eta}) \wedge \theta_{ij, k} = \log \left( \alpha \beta_i \beta_j \overline{\beta_k} \right) (\eta + \overline{\eta}) \wedge \theta_{ij, k}.
        \end{split}
    \end{equation*}
    Consequently,
    \begin{equation*}
        \begin{split}
            d\omega &= \sum\limits_{i, k = 1}^n a_{i,k} \log \left( \alpha \beta_i \overline{\beta_k} \right) (\eta + \overline{\eta}) \wedge \theta_{i,k}\\
            &+ 2 \Re \sum\limits_{i, j, k = 1}^n b_{ij,k} \left( \log \alpha \ \eta \wedge \overline{\eta} \wedge \theta_{ij,k} - \log \left( \alpha \beta_i \beta_j \overline{\beta_k} \right) \eta \wedge \overline{\eta} \wedge \theta_{ij,k} \right) \\
            &= \sum\limits_{i, k = 1}^n a_{i,k} \log \left( \alpha \beta_i \overline{\beta_k} \right) (\eta + \overline{\eta}) \wedge \theta_{i,k} - 2 \Re \sum\limits_{i, j, k = 1}^n b_{ij,k} \log \left( \beta_i \beta_j \overline{\beta_k} \right) \ \eta \wedge \overline{\eta} \wedge \theta_{ij,k}.
        \end{split}
    \end{equation*}

    As $\theta_{i, k}$ is $(n-1, n-1)$, $\theta_{ij, k}$ is $(n-2, n-1)$ and $J (\Re \gamma) = i \Im (J \gamma)$ for any form $\gamma$, we have
    \[
        Jd\omega = \sum\limits_{i, k = 1}^n a_{i,k} \log \left( \alpha \beta_i \overline{\beta_k} \right) (\overline{\eta} - \eta) \wedge \theta_{i,k} + 2i \Im \sum\limits_{i, j, k = 1}^n b_{ij,k} \log \left( \beta_i \beta_j \overline{\beta_k} \right) \ \eta \wedge \overline{\eta} \wedge \theta_{ij,k}.
    \]

    Note that the second term is closed, so we finally get 
    \begin{equation*}
        \begin{split}
            dJd\omega &= 2 \log \alpha \sum\limits_{i, k = 1}^n a_{i,k} \log \left( \alpha \beta_i \overline{\beta_k} \right) \overline{\eta} \wedge \eta \wedge \theta_{i,k} - \sum\limits_{i, k = 1}^n a_{i,k} \left( \log \left( \alpha \beta_i \overline{\beta_k} \right) \right)^2 (\overline{\eta} - \eta) \wedge (\eta + \overline{\eta}) \wedge \theta_{i, k} \\
            &= 2 \sum\limits_{i, k = 1}^n a_{i,k} \log \left( \alpha \beta_i \overline{\beta_k} \right) \left( \log \left( \alpha \beta_i \overline{\beta_k} \right) - \log \alpha \right) \eta \wedge \overline{\eta} \wedge \theta_{i,k} \\
            &= 2 \sum\limits_{i, k = 1}^n a_{i,k} \log \left( \alpha \beta_i \overline{\beta_k} \right) \log \left(\beta_i \overline{\beta_k} \right) \eta \wedge \overline{\eta} \wedge \theta_{i,k}.
        \end{split}
    \end{equation*}
    It follows that $\Omega$ is astheno-\K \ if and only if $a_{i,k} \log \left( \alpha \beta_i \overline{\beta_k} \right) \log \left(\beta_i \overline{\beta_k} \right) = 0$ for all $i, k = \overline{1, n}$.

    But then, exactly as in the pluriclosed case, $a_{jj} \neq 0$ means that for any $1 \le j \le n$, $| \beta_j |^2 = 1 $ or $\alpha |\beta_j|^2 = 1$ \ie there is one $i_0$ such that $\alpha |\beta_{i_0}|^2 = 1$ and $|\beta_j|^2 = 1$ for all other $j \neq i_0$.

    \medskip

    The implication $(4) \Rightarrow (2)$ is straightforward, since it implies that the roots of $f_0$ are $\alpha$, $\beta_{i_0}$, $\overline{\beta}_{i_0}$ satisfying $\alpha \beta_{i_0} \overline{\beta}_{i_0}=1$ and the roots of $h$ have norm 1. Conversely, if (2) is satisfied, take any $\beta_i$ such that $|\beta_i|=1$. Then the minimal polynomial of $\beta_i$ cannot have $\alpha$ as a root, since the roots of this minimal polynomial come in pairs with their inverse. Therefore, the minimal polynomial $f_0$ of $\alpha$ is of degree 3 and has the roots $\alpha$, $\beta_{i_0}$, $\overline{\beta}_{i_0}$. This further implies  $P_M$ splitting as $f_0h$, where $h$ has the roots $\{\beta_i\}_{i \neq 0}$ and since they come in pairs with their inverses, $h$ is a self-reciprocal polynomial.

    \medskip

    Finally, the equivalence $(2) \iff (5)$ is immediate using \ref{nrbetti}. Indeed, we have $h^2(T_M) = g_2$ and, because $M$ is diagonalizable, this is the algebraic multiplicity of $M^{\wedge 2}$ \ie the number of pairs of eigenvalues whose product is $1$. By the properties of $M$, this can only happen if $n-1$ of the eigenvalues $\beta_1, ..., \beta_n$ are of norm $1$.
\end{proof}

\bigskip

The lemma below illustrates the usefulness of the characterization found in \ref{th:Mdiag} for generating examples in any dimension.

\begin{lemma}
    Take $P$ an integer monic polynomial of degree $2n + 1$ such that $P = f_0h$ in $\mathbb{Z}[X]$ where $f_0$ has degree $3$, one real root $\alpha > 0$, $\alpha \neq 1$, $f_0(0) = -1$, while $h$ has purely complex roots and is self-reciprocal. Assume further that the unique factorization of $h$ contains only polynomials without multiple roots (this happens, for instance, if $h$ is irreducible).

    Then there is a diagonalizable matrix $M \in \textrm{SL}_{2n + 1}(\Z)$ satisfying all the properties required for the Endo-Pajitnov construction such that $P_M = P$. In particular, using \ref{th:Mdiag}, there is complex compact manifold $T_M$ of dimension $n + 1$ that carries both a pluriclosed and an astheno-\K \ metric that is not of OT type.
\end{lemma}

\begin{proof}
    Simply take $M$ to be made out of diagonal blocks representing the companion matrices of the irreducible factors of $P$. Since they all have simple roots, all these companion matrices are diagonalizable, thus $M$ is diagonalizable. 
\end{proof}
        
\begin{example}
    Starting from the polynomial $$P(X) = X^7 + 3X^5 - X^4 + 3X^3 - 2X^2 + X - 1 = (X^3 + X - 1)(X^2 + 1)^2,$$ the corresponding $T_M$ constructed as above is a $4$-dimensional complex manifold, not of OT type, that is both pluriclosed and astheno-\K.
\end{example}

\bigskip

While we now have a complete characterization for the existence of pluriclosed metrics on Endo-Pajitnov manifolds for a diagonalizable $M$, we end by proving below a companion result stating that the situation from \ref{th:Mdiag} is the only one allowing the existence of such a metric:

\begin{theorem}
    \label{th:Mnediag}
    Let $T_M$ be the manifold associated to a matrix $M \in \textrm{SL}(2n+1, \mathbb{Z})$. The following are equivalent:
    \begin{enumerate}[(1)]
        \item $T_M$ admits a pluriclosed metric.
        \item $M$ is diagonalizable and there is $1 \leq i_0 \leq n$ such that $\alpha\beta_{i_0}\overline{\beta}_{i_0}=1$ and $|\beta_i|=1$, for any $1 \leq i \leq n$, $i \neq i_0$.
    \end{enumerate}
\end{theorem}

\begin{proof}
    Considering \ref{th:Mdiag}, we only have to prove $(1) \Rightarrow (2)$. Take $\omega$ a pluriclosed metric on $T_M$. Again via an averaging procedure, we can perform the computations in the Lie algebra,
    \begin{equation*}
        \omega = \sum\limits_{i = 1}^n c_i \eta \wedge \overline{\theta_i} + \sum\limits_{i = 1}^n d_i \overline{\eta} \wedge \theta_i + \sum\limits_{i, j = 1}^n \overline{a_{ij}} \theta_i \wedge \overline{\theta_j} + b \eta \wedge \overline{\eta},
    \end{equation*}
    with $A = (a_{ij})_{ij}$ a Hermitian-antisymmetric matrix. Denote 
    \[
    P = 2\overline{\Delta}^T A \Delta + (\overline{\Delta}^T)^2 A + A \Delta^2 + \log \alpha( \overline{\Delta}^T A + A\Delta).
    \]
    where as before $\Delta = \log(R)$, so is an upper-triangular matrix. By \eqref{pluriclosed}, the pluriclosed condition is equivalent to $P = 0$.

    By direct computation, $P_{11} = a_{11} \log(\alpha \beta_1 \overline{\beta_1}) \log (\beta_1 \overline{\beta_1})$. Because $\omega$ is positive definite, $a_{11} \neq 0$, so at least for $\beta_1$ we have that either $\alpha \beta_1 \overline{\beta_1} = 1$ or $|\beta_1| = 1$.

    But by \ref{rmk:nudepdeformaJordan}, as the biholomorphism class of $T_M$ does not depend on the choice of Jordan form $R$, the order of the complex eigenvalues is arbitrary, so we in fact have that for any $1 \le i \le n$, $\alpha \beta_i \overline{\beta_i} = 1$ or $|\beta_i| = 1$. This can only happen if there is one $i_0$ such that $\alpha |\beta_{i_0}|^2 = 1$ and $|\beta_j|^2 = 1$ for all other $j \neq i_0$.

    To prove that $M$ is diagonalizable \ie $\Delta$ is diagonal, we use the same trick. Assume that it is not; then we can choose an ordering of the eigenvalues such that $\alpha \beta_1 \overline{\beta_1} = 1$, $| \beta_i | = 1$ for all $2 \le i \le n$ and $\beta_2 = \beta_3$. Then
    \[
    \Delta = 
    \begin{pmatrix}
        \log \beta_1 & 0 & 0 & \dots & 0 \\
        0 & \log \beta_2 & \Delta_{23} & \dots & 0 \\
        0 & 0 & \log \beta_2 & \dots & 0 \\
        \dots & \dots & \dots & \dots & \dots 
    \end{pmatrix},    
    \]
    where $\Delta_{23} \neq 0$ can be determined explicitly from the formula of the logarithm of a Jordan block (see \eg \cite[Section 11.3]{higham}).

    On the other hand, by direct computation, 
    \[
    P_{23} = \log (\alpha \beta_2 \overline{\beta_2}) \left( a_{23} \log(\beta_2 \overline{\beta_2}) + a_{22} \Delta_{23} \right) =  a_{22} \Delta_{23} \log \alpha.
    \]
    As $a_{22} \neq 0$, it follows that $\Delta_{23} = 0$, a contradiction.
\end{proof}

\bigskip

\textbf{Acknowledgments.} We thank Liviu Ornea for many helpful discussions along the way.

\end{document}